\documentclass[A4paper, leqno]{amsart}      

\usepackage{amsthm}
\usepackage{amsfonts}
\usepackage{amsmath}
\usepackage{amssymb}
\usepackage{verbatim, color}
\usepackage{enumerate}
\usepackage{pictex}
\usepackage{url}
\usepackage{verbatim, color}
\usepackage{latexsym}
\usepackage[latin1]{inputenc}

\numberwithin{equation}{section}

\begin{document}

 \newcounter{thlistctr}
 \newenvironment{thlist}{\
 \begin{list}%
 {\alph{thlistctr}}%
 {\setlength{\labelwidth}{2ex}%
 \setlength{\labelsep}{1ex}%
 \setlength{\leftmargin}{6ex}%
 \renewcommand{\makelabel}[1]{\makebox[\labelwidth][r]{\rm (##1)}}%
 \usecounter{thlistctr}}}%
 {\end{list}}

\thispagestyle{empty}

\newtheorem{Lemma}{\bf LEMMA}[section]
\newtheorem{Theorem}[Lemma]{\bf THEOREM}
\newtheorem{Claim}[Lemma]{\bf CLAIM}
\newtheorem{Corollary}[Lemma]{\bf COROLLARY}
\newtheorem{Proposition}[Lemma]{\bf PROPOSITION}
\newtheorem{Example}[Lemma]{\bf EXAMPLE}
\newtheorem{Fact}[Lemma]{\bf FACT}
\newtheorem{definition}[Lemma]{\bf DEFINITION}
\newtheorem{Notation}[Lemma]{\bf NOTATION}
\newtheorem{remark}[Lemma]{\bf REMARK}

\newcommand{\restrict}{\mbox{$\mid\hspace{-1.1mm}\grave{}$}}
\newcommand{\covers}{\mbox{$>\hspace{-2.0mm}-{}$}}
\newcommand{\covered}{\mbox{$-\hspace{-2.0mm}<{}$}}
\newcommand{\notcover}{\mbox{$>\hspace{-2.0mm}\not -{}$}}

\newcommand{\boldalpha}{\mbox{\boldmath $\alpha$}}
\newcommand{\boldbeta}{\mbox{\boldmath $\beta$}}
\newcommand{\boldgamma}{\mbox{\boldmath $\gamma$}}
\newcommand{\boldxi}{\mbox{\boldmath $\xi$}}
\newcommand{\boldlambda}{\mbox{\boldmath $\lambda$}}
\newcommand{\boldmu}{\mbox{\boldmath $\mu$}}

\newcommand{\barzero}{\bar{0}}

\newcommand{\sfq}{{\sf q}}
\newcommand{\sfe}{{\sf e}}
\newcommand{\sfk}{{\sf k}}
\newcommand{\sfr}{{\sf r}}
\newcommand{\sfc}{{\sf c}}
\newcommand{\restr}{\negmedspace\upharpoonright\negmedspace}

{\color{red}
\numberwithin{equation}{section}

\newtheorem{definicion}{Definition}[section]
\newtheorem{Definition}[definicion]{Definition}
\newtheorem{theorem}[definicion]{Theorem}
\newtheorem{corollary}[definicion]{Corollary}
\newtheorem{lema}[definicion]{Lemma}
\newtheorem{notacion}[definicion]{Notation}
\newtheorem{teorema}[definicion]{Theorem}
\newtheorem{corolario}[definicion]{Corollary}
\newtheorem{ejemplo}[definicion]{Example}
\newtheorem{observacion}[definicion]{Observation}
\newtheorem{proposicion}[definicion]{Proposition}
\newtheorem{Remark}[definicion]{Remark}
\newtheorem{proposition}[definicion]{Proposition}

\newenvironment{Proof}{\noindent\bf Proof \rm}{$\hfill
\square$}

\newcommand{\debaj}[2]{ #1 \to_H #2}

\newcounter{ecuacionDef} \setcounter{ecuacionDef}{0}
\renewcommand{\theecuacionDef}{\arabic{ecuacionDef}}

\newenvironment{ecuacionDef}[1]%
{
	\refstepcounter{ecuacionDef}
	\vspace{0.12cm}
	{\rm (E\theecuacionDef)}:  #1
	\vspace{0.12cm}
}%

\newcounter{numeroAxioma} \setcounter{numeroAxioma}{0}
\renewcommand{\thenumeroAxioma}{\arabic{numeroAxioma}}

\newenvironment{numeroAxioma}[1]%
{
	\refstepcounter{numeroAxioma}
	\vspace{0.12cm}
	(A\thenumeroAxioma):  #1
	\vspace{0.12cm}
}%

}

\title[A Converse to a Katri\v{n}\'{a}k's Theorem] {Regular Double $p$-Algebras: A converse to a Katri\v{n}\'{a}k Theorem, and Applications}

\author[J. M. CORNEJO, M. Kinyon \and H. P. SANKAPPANAVAR]{Juan M. CORNEJO*, Michael KINYON** \and \\Hanamantagouda P. SANKAPPANAVAR***}

\newcommand{\acr}{\newline\indent}

\address{\llap{*\,} Departamento de Matem\'atica\acr
Universidad Nacional del Sur\acr
Alem 1253, Bah\'ia Blanca, Argentina\acr
INMABB - CONICET}

\email{jmcornejo@uns.edu.ar}

\address{\llap{**\,}Department of Mathematics\acr
                             University of Denver\acr
                             Denver,  Colorado, 80208 \acr
                             USA.}

\email{mkinyon@math.du.edu}

\address{\llap{***\,}Department of Mathematics\acr
                              State University of New York\acr
                              New Paltz, New York, 12561\acr
                              USA.}

\email{sankapph@hawkmail.newpaltz.edu}
            
\dedicatory{\bf In Memory of Professor Tibor Katri\v{n}\'{a}k}

\subjclass[2010]{Primary:03G25, 06D20, 08B15, 06D15, 03C05, 03B50; \, Secondary: 08B26, 06D30, 06E75}
\keywords{regular double $p$-algebra, regular dually pseudocomplemented Heyting algebra, regular pseudocomplemented dual Heyting algebras, and regular double Heyting algebra, regular De Morgan p-algebras, regular De Morgan Heyting algebras, regular De Morgan double Heyting algebras, regular De Morgan double p-algebras, logic $\mathcal{RDPCH }$, logic $\mathcal{RPCH}^d$, logic 
$\mathcal{RDMH}$ }
\date{}

\maketitle

\begin{abstract}

In 1973, Katri\v{n}\'{a}k proved that regular double $p$-algebras can be regarded as (regular) double Heyting algebras by ingeniously constructing binary terms for the Heying implication and its dual in terms of pseudocomplement and its dual.  In this paper 
we prove a converse to the Katri\v{n}\'{a}k's theorem, in the sense that in the variety $\mathbb{RDPCH}$ of regular dually pseudocomplemented Heyting algebras, the implication operation $\to$ satisfies the Katrinak's formula.   As applications of this result together with the above-mentioned Katri\v{n}\'{a}k's theorem, we show that the varieties $\mathbb{RDBLP}$, $\mathbb{RDPCH}$,  $\mathbb{RPCH}^d$ and $\mathbb{RDBLH}$ of regular double p-algebras, regular dually pseudocomplemented Heyting algebras, regular pseudocomplemented dual Heyting algebras, and regular double Heyting algebras, respectively, are term-equivalent to each other and also that the varieties $\mathbb{RDMP}$, $\mathbb{RDMH}$, $\mathbb{RDMDBLH}$, $\mathbb{RDMDBLP}$ of regular De Morgan p-algebras, regular De Morgan Heyting algebras, regular De Morgan double Heyting algebras, and regular De Morgan double p-algebras, respectively, are also term-equivalent to each other.  From these results and recent results of \cite{AdSaVc19}  and \cite{ AdSaVc20}, we deduce that the lattices of subvarieties of all these varieties have cardinality $2^{\aleph_0}$.  We then define new logics, $\mathcal{RDPCH }$, $\mathcal{RPCH}^d$, and $\mathcal{RDMH}$, and show that they are algebraizable with $\mathbb{RDPCH}$,  $\mathbb{RPCH}^d$ and $\mathbb{RDMH}$, respectively as their equivalent algebraic semantics.  It is also deduced that the lattices of extensions of all of the above mentioned logics have cardinality $2^{\aleph_0}$. 
\end{abstract}

\thispagestyle{empty}

\section{Introduction} \label{SA}

It is well-known that slight weakenings of complement in a Boolean algebra have led to the notions of pseudocomplemented lattice, dually pseudocomplemented lattice and De Morgan algebra.

A bounded lattice
$L$ is pseudocomplemented if and only if for every $x \in L$, there exists a largest element $x^* \in L$ such that $x \land x^* =0$.
 
An algebra $\langle L, \lor, \land, ^*, 0, 1\rangle$ is called a $p$-algebra if $\langle L, \lor, \land, 0,1\rangle$ is a bounded distributive lattice
and $^*$ is a pseudocomplementation on L.  A dual $p$-algebra is, of course, defined dually.  

 In 1949, Ribenboim \cite{Ri49} gave the following equational characterization of the class  of distributive pseudocomplemented lattices:

 (R1)  $(x \lor y) \lor  z \approx x \lor (y \lor z),$

(R2)  $(x \land y) \land z \approx x \land (y \land z),$

(R3) $x \lor y \approx y \lor x,$

(R4) $x \land y \approx y \land x,$

(R5) $ x \lor x \approx x,$

(R6) $ x \land x \approx x,$ 

(R7)  $x \lor (x \land y) \approx x,$

(R8)  $x \land (y \lor z) \approx (x \land y) \lor (x \land z),$

(R9) $x \land x^* \approx y \land y^*,$

(R10) $x \land (x \land y)^* \approx x \land y^*,$

(R11)  $x \land (x \land x^*)^* \approx x,$

(R12)  $(x \land x^*)^{**} \approx x \land  x^*$.\\
 
It follows from this result of Ribenboim that the class of distributive $p$-algebras is a variety. 
The following equational basis for distributive $p$-algebras is given much later in \cite[Corollary 2.8]{Sa87b} and will be useful in the sequel. \\

\begin{definition}
An algebra $\mathbf A =\langle A, \lor. \land, ^*,  0, 1\rangle$ is a $p$-algebra if the following conditions are satisfied:
\begin{itemize}
\item[\rm(S1)]  $\langle A, \lor. \land, 0, 1\rangle$ is a bounded distributive lattice,
\item[\rm(S2)]  $(x \lor y)^* \approx x^* \land y^*$,
\item[\rm(S3)]  $(x \land y)^{**} \approx x^{**} \land y^{**}$,
\item[\rm(S4)]  $ x \leq x^{**}$,
\item[(S5)] $x^* \land x^{**} \approx 0$ \textup. 
\end{itemize}
\end{definition}

Note that the identity (S5) can be replaced by the identity:  $x \land x^* \approx 0$.

Combining the notions of p-algebra and its dual, one naturally obtains double p-algebras which were first introduced in 1972 by Varlet \cite{Va72}.

\begin{definition}
An algebra $\mathbf A =\langle A, \lor. \land, ^*, ^+, 0, 1\rangle$ is a double $p$-algebra if the following conditions are satisfied:
\begin{enumerate}[\rm(1)]
	\item $\langle A, \lor. \land, ^*, 0, 1\rangle$ is a $p$-algebra,
	\item $\langle A, \lor. \land, ^+, 0, 1\rangle$ is a dual $p$-algebra.
\end{enumerate}
\end{definition}

Let $\mathbb{DBLP}$ denote the variety of double $p$-algebras.


The notion of regularity for double $p$-algebras was also first introduced by Varlet \cite{Va72}.

A double $p$-algebra $\mathbf A$ is {\it regular} if $\mathbf A$ satisfies the following quasi-identity:
\begin{equation} \tag{R}
    x^*=y^* \text{ and } x^+=y^+ \text{  implies  } x=y.  
\end{equation}
The class of regular algebras in $\mathbb{DBLP}$ is denoted by $\mathbb{RDBLP}.$

Varlet \cite{Va72} and Katri\v{n}\'{a}k \cite{Ka73} have  
proved that the regular double $p$-algebras form a subvariety of 
$\mathbb{DBLP}$, defined by the identity:
\begin{equation} \tag{M}
    (x \wedge x^+) \vee (y \vee y^{\ast}) \approx y \vee y^{\ast}.\\
\end{equation}

Brouwer's intuitionism that questioned some of the principles accepted in classical mathematics like the law of excluded middle and that of double negation, as is well-known, led to Heyting algebras, which, in turn, have been recently \cite{Sa08} generalized to  semi-Heyting algebras.  It turns out that semi-Heyting algebras are also closely related to $p$-algebras.   

An algebra $\mathbf{A} = \langle A,\lor, \land, \to,  0, 1)$ is a {\it semi-Heyting algebra} if ${\mathbf A}$ satisfies the following conditions:
\begin{enumerate}[{\rm(SH1)}]
  \item $ (A,\lor, \land, 0, 1)$ is a bounded (distributive) lattice,
  \item ${x \land ( x \to y )\approx x \land y}$,  \label{conditionH2}
  \item ${x \land ( y \to z) \approx x \land [(x \land y) \to (x \land z)]}$, \label{conditionH3}
  \item ${(x \to x\approx 1}$.
\end{enumerate}
A semi-Heyting algebra $\mathbf{A}$ is a {\it Heyting algebra} if ${\mathbf A}$ satisfies the following condition: 

(H) ${(x \land y) \to x\approx 1}$.

\medskip
Let $\mathbb{SH}$ and $\mathbb{H}$ denote, respectively, the varieties of semi-Heyting algebras and Heyting algebras.
It is well-known \cite{Sa08} that if $\mathbf A$ is a semi-Heyting algebra, then $ \langle A, \lor, \land, ^*, 0,1\rangle$ is a distributive p-algebra, where $x^*:= x \to 0$.

\medskip
Combining the notions of Heyting algebra and its dual, we naturally arrive at the notion of a double (or bi-)
Heyting algebra.

\begin{definition}
An algebra ${\bf A} = \langle A,\lor, \land, \to, \leftarrow, 0, 1\rangle$ is a
{\it double Heyting  {\rm(}bi-Heyting, Heyting-Brouwer{\rm)} algebra} if 
\begin{enumerate}[\rm(i)]
	\item $\langle A,\lor, \land, \to, 0, 1\rangle$ is a
	Heyting algebra, and
	\item $\langle A, \land, \lor, \leftarrow, 0, 1\rangle$ is a
	dual Heyting algebra.
\end{enumerate}
\end{definition}

We denote by $\mathbb{DBLH}$ the variety of double Heyting algebras.
If $\mathbf A$ is a double Heyting algebra, then it is well-known that $\langle A, \lor, \land, ^*, ^+, 0,1\rangle$ is a distributive double p-algebra, where $x^*:= x \to 0$ and $x^+:= x \leftarrow 1.$

\smallskip
 An algebra in $\mathbb{DBLH}$ is {\it regular} if it satisfies (M). 
 Let $\mathbb{RDBLH}$ 
 denote the  variety of regular double Heyting algebras.  

Katri\v{n}\'{a}k \cite{Ka73} proved that a regular double p-algebra  can be regarded as a (regular) double Heyting algebra  
 by ingeniously constructing the following binary terms for the Heying implication and its dual:
 \[ \quad \kappa(x,y) := (x^{\ast} \lor y^{\ast \ast}){^{\ast \ast}} \land [(x \lor x^{\ast}){^+} \lor x^{\ast} \lor y \lor y^{\ast}], \]
  \[ \quad \kappa^d(x,y) := (x^{+} \land y^{++}){^{++}} \lor [(x \lor x^{+}){^*} \land x^{+} \land y \land y^{+}]. \] 
We refer to these two binary terms as {\it Katri\v{n}\'{a}k's term} and {\it dual Katri\v{n}\'{a}k's term}, respectively. 
 
 More precisely, Katri\v{n}\'{a}k proved the following theorem (which is the theorem referred to in the title of this paper):
 
 \begin{Theorem} {\rm (Katri\v{n}\'{a}k \cite{Ka73})} \label{K_Theorem}
Let ${\mathbf A} = \langle A,\lor, \land, ^*, ^+,  0, 1)$ be a regular double p-algebra and $a,b \in A$.  
Then ${\mathbf {A_{dbl}}} := \langle A,\lor, \land, \to, \leftarrow, 0, 1)$ is a (regular)  double Heyting algebra, where the operations $\to$ and $\leftarrow$ are defined as follows:
\[a \to b := \kappa(a,b) \quad \mbox{   and    }  \quad a \leftarrow b := \kappa^d (a, b). \]
\end{Theorem}

Double p-algebras and double Heyting algebras are closely related to dually pseudocomplemented Heyting algebras-- the study of the latter was initiated in \cite{Sa85} and continued in \cite{Sa12, Sa14, Sa14a}.

An algebra ${\mathbf A} = \langle A,\lor, \land, \to, ^+, 0, 1)$ is a {\it dually pseudocomplemented Heyting algebra} if ${\mathbf A}$ satisfies the following conditions:
\begin{enumerate}
  \item[{\rm(a)}] $ \langle A,\lor, \land, \to, 0, 1\rangle$ is a Heyting algebra,
\item[{\rm(b)}] $\langle A,\lor, \land, ^+, 0, 1)$ is a dual $p$-algebra.
\end{enumerate}

The variety of dually pseudocomplemented Heyting algebras is denoted by $\mathbb{DPCH}$. 
Let $\mathbb{A} \in \mathbb{DPCH}$ and $x \in \mathbb{A}.$  Then we define $x^*:= x \to 0.$   An algebra $\mathbb{A} \in \mathbb{DPCH}$ is {\it regular} if it satisfies (M).
$\mathbb{RDPCH}$ denotes the subvariety of $\mathbb{DPCH}$ consisting of regular dually pseudocomplemented Heyting algebras.

We also denote by $\mathbb{RPCH}^d$ the variety consisting of regular pseudocomplemented dual Heyting algebras.  Note that $\mathbb{RPCH}^d$ consists of algebras dual to the members of $\mathbb{RDPCH}$. \\

Observe that:

\begin{enumerate}[(1)]
	\item If $\mathbf A \in \mathbb{RDBLP}$, then $\langle A, \lor, \land, \to, ^+, 0,1\rangle$ is a regular dually pseudocomplemented Heyting algebra, where $\to$ is as given in Theorem \ref{K_Theorem}.
	
	\item If $\mathbf A \in \mathbb{RDBLP}$, then $\langle A, \lor, \land, \to, \leftarrow, 0,1\rangle$ is a regular double Heyting algebra, where $\to$ and $\leftarrow $ are as given in Theorem \ref{K_Theorem}.
	
	\item If $\mathbf A \in \mathbb{RDPCH}$, then $\langle A, \lor, \land, ^*, ^+, 0,1\rangle$ is a regular double p-algebra. 
	
	
	\item  If $\mathbf A \in \mathbb{RDBLH}$, then $\langle A, \lor, \land, ^*, ^+, 0,1\rangle$ is a regular double $p$-algebra, where $x^*:= x \to 0$  and  $x^+:= x \leftarrow 1.$
	
	\item  If $\mathbf A \in \mathbb{RDBLH}$, then $\langle A, \lor, \land, \to, ^+, 0,1\rangle$ is a regular dually pseudocomplemented Heyting algebra, where $x^+:= x \leftarrow 1.$
	
	\item  If $\mathbf A \in \mathbb{RPCH}^d$, then $\langle A, \lor, \land, ^*,  ^+, 0,1\rangle$ is a regular double $p$-algebra, where $x^+:= x \leftarrow 1.$
	
\end{enumerate}

De Morgan p-algebras were introduced in \cite{Ro81} and further investigated in \cite{Sa86},  \cite{Sa87c} and \cite{Sa14b} as pseudocomplemented De Morgan algebras. 

An algebra $\mathbf A= \langle A, \lor,\land, ^*, ', 0, 1\rangle$ is a De Morgan p-algebra (or pseudocomplemented De Morgan algebra) if 
\begin{enumerate}[(i)]
	\item  $ \langle A, \lor,\land, ^*, 0, 1\rangle$ is a distributive p-algebra,
	\item $ \langle A, \lor,\land, ', 0, 1\rangle$ is a De Morgan algebra.
\end{enumerate}

De Morgan Heyting algebra were introduced in \cite{Sa87a}.

An algebra $\mathbf A= \langle A, \lor,\land, \to, ', 0, 1\rangle$ is a De Morgan Heyting algebra if 
\begin{enumerate}[(i)]
	\item $\langle A, \lor,\land, \to, 0, 1\rangle$ is a Heyting algebra,
	\item  $\langle A, \lor,\land, ', 0, 1\rangle$ is a De Morgan algebra.
\end{enumerate}

We now introduce a new variety of algebras called De Morgan double $p$-algebras.

An algebra $\mathbf A =\langle A, \lor, \land, ^*, ^+, ', 0, 1\rangle$ is a De Morgan double $p$-algebra if the following conditions are satisfied:
\begin{enumerate}[\rm(1)]
	\item $\langle A, \lor. \land, ^*, ^+, 0, 1\rangle$ is a double $p$-algebra,
	\item $\langle A, \lor, \land, ', 0, 1\rangle$ is a De Morgan algebra.
\end{enumerate}

Next, we also introduce another new variety of algebras called De Morgan double Heyting algebras.

An algebra $\mathbf A= \langle A, \lor,\land, \to,  \leftarrow, ', 0, 1\rangle$ is a De Morgan double Heyting algebra if 
\begin{enumerate}[(i)]
	\item $\langle A, \lor,\land, \to,   \leftarrow, 0, 1\rangle$ is a double Heyting algebra,
	\item  $\langle A, \lor,\land, ', 0, 1\rangle$ is a De Morgan algebra.
\end{enumerate}

Let $\mathbb{DMP}$, $\mathbb{DMH}$, $\mathbb{DMDBLP}$ and $\mathbb{DMDBLH}$ denote, respectively, the varieties of De Morgan  $p$-algebras, De Morgan Heyting algebras, De Morgan  double $p$-algebras  and De Morgan double Heyting algebras.  

Romanowska \cite{Ro81} has observed the following:  If $\mathbf A \in \mathbb{DMP}$, then the algebra $\mathbf{A}^{\bf dblp} := \langle A, \lor, \land, ^*, ^+, 0, 1\rangle$, where $a^+:= a'{^*}'$, is a double $p$-algebra, $a^+:= a'{^*}'$ being the dual psedocomplement of $a \in A$.  

Similarly, we observe that if $\mathbf A \in \mathbb{DMH}$, then the algebra \\
$\mathbf{A}^{\bf dblp} := \langle A, \lor, \land, ^*, ^+, 0, 1\rangle$, where 
 $a^* := a \to 0$ and $a^+:= a'{^*}'$, is easily seen to be a double $p$-algebra, $a^+:= a'{^*}'$ being the dual pseudocomplement of $a \in A$.  We also note that if $\mathbf A \in \mathbb{DMDBLP}$, then the reduct  $\langle A, \lor, \land, ^*, ^+, 0, 1\rangle$ is a double $p$-algebra and that if $\mathbf A \in  \mathbb{DMDBLH}$, then the reduct  $\langle A, \lor, \land, ^*, ^+, 0, 1\rangle$, where $a^* := a \to 0$ and $a^+:= a \leftarrow 1$ is a double $p$-algebra.

\smallskip
Let $\mathbb{V} \in \{\mathbb{DMP}, \mathbb{DMH}, \mathbb{DMDBLP}, \mathbb{DMDBLH}\}$.
 An algebra $\mathbf{A} \in \mathbb{V}$ is {\it regular} if it satisfies:
 \begin{equation} \tag{M1}
    (x \wedge x'{^*}') \vee (y \vee y^{\ast}) \approx y \vee y^{\ast}.\\
\end{equation}

 Let $\mathbb{RDMP}$, $\mathbb{RDMH}$, $\mathbb{RDMDBLP}$ and $\mathbb{RDMDBLH}$ denote, respectively, the varieties of regular De Morgan  $p$-algebras, regular De Morgan Heyting algebras, regular De Morgan  double $p$-algebras  and regular De Morgan double Heyting algebras.

Regularity was first examined for the variety  $\mathbb{DMP}$ of De Morgan $p$-algebras (i.e., pseudocomplemented De Morgan algebras) in \cite{Sa86} and for $\mathbb{DMH}$ of De Morgan Heyting algebras in \cite{Sa87a}.

The purpose of this note is two-fold.  Firstly, we prove a converse (see Theorem \ref{Main}) to the above-mentioned  Katrin\'{a}k's theorem (Theorem \ref{K_Theorem}).    
More precisely, as our main theorem, we prove:   

(a)  If $\mathbf A := \langle A, \vee, \wedge, \to, ^+, 0, 1 \rangle 
\in  \mathbb{RDPCH}$, then
$\mathbf A \models x \to y \approx \kappa(x,y)$.

(b) If $\mathbf A := \langle A, \vee, \wedge, ^*, \leftarrow, 0, 1 \rangle 
\in  \mathbb{RPCH}^d$, then
$\mathbf A \models x \leftarrow y \approx \kappa^d (x,y)$ 

 (c) If $\mathbf A := \langle A, \vee, \wedge, \to, \leftarrow, 0, 1 \rangle 
\in  \mathbb{RDBLH}$, then	 
$\mathbf A \models x \to y \approx \kappa(x,y)$, and

\quad \  $\mathbf A \models x \leftarrow y \approx \kappa^d (x,y)$.

\noindent Secondly, we present several applications, both algebraic and logical, of Katrin\'{a}k's theorem (Theorem \ref{K_Theorem}) and our main theorem  
(Theorem \ref{Main}).

 As a first application,  
 we show that the varieties 
$\mathbb{RDBLP}$, $\mathbb{RDPCH}$, $\mathbb{RPCH}^d$, and $\mathbb{RDBLH}$ 
 are term-equivalent to each other.  Consequently, we obtain that the lattice of subvarieties $\mathcal{L}_{\rm V}(\mathbb{RDBLP})$, $\mathcal{L}_{\rm V}(\mathbb{RDPCH})$, $\mathcal{L}_{\rm V}(\mathbb{RPCH}^d)$, and $\mathcal{L}_{\rm V}(\mathbb{RDBLH})$ are isomorphic to one another.  
 
 As a second application, it is shown that the varieties $\mathbb{RDMP}$, 
 $\mathbb{RDMH}$, $\mathbb{RDMDBLH}$ and $\mathbb{RDMDBLP}$ are also term-equivalent to each other. 
 Therefore, the lattice of subvarieties $\mathcal{L}_{\rm V} (\mathbb{RDMP})$, $\mathcal{L}_{\rm V}(\mathbb{RDMH})$ 
 $\mathcal{L}_{\rm V}(\mathbb{RDMDBLH})$, and $\mathcal{L}_{\rm V}(\mathbb{RDMDBLP})$
 are isomorphic to one another.


 
As a third application, 
we deduce, from these results and recent results of \cite{AdSaVc19} and \cite{ AdSaVc20},  that each of the lattices of subvarieties of the varieties
$\mathbb{RDPCH}$, $\mathbb{RDBLH}$, $\mathbb{RPCH}^d$ and $\mathbb{RDMH}$  
 has cardinality $2^{\aleph_0}$. 

As a fourth application, we introduce new logics 
$\mathcal{RDPCH}$ and $\mathcal{RPCH}^d$ 
and show that they are algebraizable with   
$\mathbb{RDPCH}$, and $\mathbb{RPCH}^d$, 
respectively as their equivalent algebraic semantics.  It follows that 
 the logics
 $\mathcal{RDPCH}$, $\mathcal{RPCH}^d$ 
 are equivalent to each other. 
These results, when combined with some results of \cite{CoSa22}, in turn, lead us to conclude that each of the lattices of axiomatic extensions of logics 
$\mathcal{RDPCH}$ and $\mathcal{RPCH}^d$ 
has cardinality $2^{\aleph_0}$, as well.

As a fifth application, we introduce the logic  
$\mathcal{RDMH}$ and show that it is algebraizable  
with 
$\mathbb{RDMH}$ 
 as its equivalent algebraic semantics. 
 It follows that the lattice of axiomatic extensions of the logic $\mathcal{RDMH}$  
 has cardinality $2^{\aleph_0}$.

Before concluding this section, we present the following two well-known lemmas (of folklore) that will be useful in the next section.

\newpage
\begin{Lemma} \label{lemma_properties_H} 
	Let $\mathbf A 
	\in  \mathbb{H}$. Then 
\begin{enumerate}[\rm(a)]
	\item $x^* \leq x \to y$, \label{140420_05}
	\item $y \leq x \to y$,  \label{140420_07}
	\item $x \wedge y \leq z$ if and only if $x \leq y \to z$. \label{160420_01}
\end{enumerate}		
\end{Lemma}

\begin{Lemma} \label{lemma_properties_RDPCH}
Let $\mathbf A = \langle A, \vee, \wedge, \to, ^+, 0, 1 \rangle 
\in  \mathbb{DPCH}$. Then
\begin{enumerate}[\rm(a)]
	\item $(x \wedge y)^+ \approx x^+ \vee y^+$, \label{140420_02}
	\item $(x \vee y)^* \approx x^* \wedge y^*$, \label{140420_09}
	\item $x^{***} \approx x^*$, \label{140420_10}
	\item $x \leq x^{**}$, \label{160420_03}
	\item $x \vee (x \vee y)^+ \approx x \vee y^+$, \label{140420_03}
	\item $x \wedge (x \wedge y)^* \approx x \wedge y^*.$ \label{140420_11}
\end{enumerate}		
\end{Lemma}

\vspace{.5cm}
\section{A Converse to Katrin\'{a}k's theorem (Theorem \ref{K_Theorem})}    

In this section we prove our main theorem--a converse to Theorem \ref{K_Theorem}.  To achieve this goal, the following lemmas are crucial.

\begin{Lemma} \label{lemma_properties_alphabetagamma}
	Let $\mathbf A = \langle A, \vee,  \wedge, \to, ^+, 0, 1 \rangle \in \mathbb{RDPCH}$ and $a,b \in A$. Let 
$$\alpha = (a^* \vee b^{**})^{**}, \ \ \ \beta = (a \vee a^*)^+ \ \ \ \mbox{ and } \ \ \ \gamma = a^* \vee b \vee b^*.$$
	Then
\begin{enumerate}[\rm(a)]
	\item $(a \to b) \vee \beta = (a \to b) \vee a^+$, \label{140420_04}
	\item $(a \to b) \vee \gamma = (a \to b) \vee b^*$, \label{140420_06}
	\item $\alpha \wedge a^+ \leq a \to b$, \label{160420_02}
	\item $(a \to b)^+ \leq (\alpha \wedge (\beta \vee \gamma))^+$, \label{140420_01}
	\item $a \to b \leq \alpha$, \label{160420_04}
	\item $\alpha \vee (a \to b)^+ = 1$,  \label{160420_05}
	\item $(a \to b)^+  \vee \beta \vee \gamma = 1$, \label{160420_06} 
	\item $(a \to b)^+ \geq (\alpha \wedge (\beta \vee \gamma))^+$, \label{160420_07} 
	\item $(a \to b)^* \leq (\alpha \wedge (\beta \vee \gamma))^*$, \label{160420_08} 
	\item $(\alpha \wedge (\beta \vee \gamma))^* \leq (a \to b)^*$. \label{160420_09}
\end{enumerate}	
\end{Lemma}

\begin{proof}$\ $
	
\begin{itemize}
\item[(\ref{140420_04})]	
\noindent $(a \to b) \vee \beta	$
$\overset{ def\  of\   \beta 
}{=}  (a \to b) \vee (a \vee a^*)^+ $
$\overset{   \ref{lemma_properties_RDPCH} (\ref{140420_03}) 
}{=}  (a \to b) \vee ((a \to b) \vee a \vee a^*)^+ $
$\overset{   \ref{lemma_properties_H} (\ref{140420_05}) 
}{=}  (a \to b) \vee ((a \to b) \vee a)^+ $
$\overset{   \ref{lemma_properties_RDPCH} (\ref{140420_03}) 
}{=}  (a \to b) \vee a^+ $.
	
\item[(\ref{140420_06})]
\noindent $(a \to b) \vee  \gamma $
$\overset{ def\  of\   \gamma 
}{=}  (a \to b) \vee a^* \vee b \vee b^*  $
$\overset{   \ref{lemma_properties_H} (\ref{140420_05}) 
}{=}  (a \to b) \vee b \vee b^* $
$\overset{   \ref{lemma_properties_H} (\ref{140420_07}) 
}{=}  (a \to b) \vee b^* $.

\item[(\ref{160420_02})]
First observe that 
\noindent $(a^* \vee b^{**})^{**} $
$\overset{   \ref{lemma_properties_RDPCH} (\ref{140420_09}) 
}{=}  (a^{**} \wedge b^{***})^{*} $
$\overset{   \ref{lemma_properties_RDPCH} (\ref{140420_10})
}{=}  (a^{**} \wedge b^{*})^{*} $. 
Hence
\begin{equation} \label{equation_160420_01}  
(a^* \vee b^{**})^{**} = (a^{**} \wedge b^{*})^{*} .
\end{equation}	

 Observe that
$$
\begin{array}{lcll}
a \wedge a^+ \wedge \alpha & = & a \wedge a^+ \wedge (a^* \vee b^{**})^{**} & \mbox{definition of } \alpha \\
& = & a \wedge a^+ \wedge (a^{**} \wedge b^{*})^{*} & \mbox{by (\ref{equation_160420_01})} \\
& = & a \wedge a^+ \wedge (a \wedge a^{**} \wedge b^{*})^{*}  & \mbox{by Lemma \ref{lemma_properties_RDPCH} (\ref{140420_11})} \\
& = & a \wedge a^+ \wedge (a  \wedge b^{*})^{*} & \mbox{by Lemma \ref{lemma_properties_RDPCH} (\ref{160420_03})} \\
& = & a \wedge a^+ \wedge b^{**} & \mbox{by Lemma \ref{lemma_properties_RDPCH} (\ref{140420_11})} \\
& \leq  & b & \mbox{by Lemma \ref{lemma_properties_RDPCH} (\ref{160420_03})} 
\end{array}
$$
Hence
$a^+ \wedge \alpha \leq a \to b$ by Lemma \ref{lemma_properties_H} (\ref{160420_01}).

	\item[(\ref{140420_01})] 	
 
\noindent $	$
$\overset{  
}{}  (a \to b) \vee (\alpha \wedge (\beta \vee \gamma))^+ $
$\overset{   \ref{lemma_properties_RDPCH} (\ref{140420_02}) 
}{=}  (a \to b) \vee \alpha^+ \vee (\beta \vee \gamma)^+ $
$\overset{   \ref{lemma_properties_RDPCH} (\ref{140420_03}) 
}{=}  (a \to b) \vee \alpha^+ \vee ((a \to b) \vee \beta \vee \gamma)^+ $
$\overset{  (\ref{140420_04}) 
}{=}  (a \to b) \vee \alpha^+ \vee ((a \to b) \vee a^+ \vee \gamma)^+ $
$\overset{  (\ref{140420_06}) 
}{=}  (a \to b) \vee \alpha^+ \vee ((a \to b) \vee a^+ \vee b^*)^+ $
$\overset{   \ref{lemma_properties_RDPCH} (\ref{140420_03}) 
}{=}  (a \to b) \vee \alpha^+ \vee (a^+ \vee b^*)^+ $
$\overset{   \ref{lemma_properties_RDPCH} (\ref{140420_02}) 
}{=}  (a \to b) \vee (\alpha \wedge (a^+ \vee b^*))^+ $
$\overset{  
}{=}  (a \to b) \vee ((\alpha \wedge a^+) \vee (\alpha \wedge b^*))^+ $
$\overset{  
}{=}  (a \to b) \vee ((\alpha \wedge a^+) \vee ((a^* \vee b^{**})^{**} \wedge b^*))^+ $
$\overset{  (\ref{equation_160420_01}) 
}{=}  (a \to b) \vee ((\alpha \wedge a^+) \vee ((a^{**} \wedge b^{*})^{*} \wedge b^*))^+ $
$\overset{   \ref{lemma_properties_RDPCH} (\ref{140420_11}) 
}{=}  (a \to b) \vee ((\alpha \wedge a^+) \vee (b^* \wedge a^{***}))^+ $
$\overset{   \ref{lemma_properties_RDPCH} (\ref{140420_10}) 
}{=}  (a \to b) \vee ((\alpha \wedge a^+) \vee (b^* \wedge a^{*}))^+ $
$\overset{   \ref{lemma_properties_RDPCH} (\ref{140420_03}) 
}{=}  (a \to b) \vee ((a \to b) \vee (\alpha \wedge a^+) \vee (b^* \wedge a^{*}))^+ $
$\overset{  (\ref{160420_02}) 
}{=}  (a \to b) \vee ((a \to b)  \vee (b^* \wedge a^{*}))^+ $
$\overset{   \ref{lemma_properties_H} (\ref{140420_05}) 
}{=}  (a \to b) \vee (a \to b)^+ $
$\overset{ 
}{=}  1 $.
Hence $(a \to b)^+ \leq (\alpha \wedge (\beta \vee \gamma))^+$.

\item[(\ref{160420_04})]
\noindent $(a \to b) \wedge \alpha $
$\overset{ def\  of\   \alpha 
}{=}  (a \to b) \wedge (a^* \vee b^{**})^{**} $
$\overset{  (\ref{equation_160420_01}) 
}{=}  (a \to b) \wedge  (a^{**} \wedge b^{*})^{*} $
$\overset{   \ref{lemma_properties_RDPCH} (\ref{140420_11}) 
}{=}  (a \to b) \wedge  [(a \to b) \wedge a^{**} \wedge b^{*}]^{*}  $
$\overset{
}{=}  (a \to b) \wedge  [b^{*} \wedge (a \to b) \wedge a^{**} ]^{*}  $
$\overset{  (H\ref{conditionH3}) 
}{=}  (a \to b) \wedge  [b^{*} \wedge \{(a \wedge b^{*}) \to (b \wedge b^{*})\} \wedge a^{**}]^{*} $
$\overset{  
}{=}  (a \to b) \wedge  [b^{*} \wedge \{(a \wedge b^{*}) \to 0\} \wedge a^{**}]^{*} $
$\overset{  
}{=}  (a \to b) \wedge  [b^{*} \wedge \{(a \wedge b^{*}) \to (0 \wedge b^{*})\} \wedge a^{**}]^{*} $
$\overset{  (H\ref{conditionH3}) 
}{=}  (a \to b) \wedge [b^{*} \wedge (a \to 0) \wedge a^{**}]^{*} $
$\overset{  
}{=} (a \to b) \wedge (b^* \wedge a^{*} \wedge a^{**})^{*} $
$\overset{  
}{=}  (a \to b) \wedge 0^{*} $
$\overset{  
}{=}  (a \to b) $.

\item[(\ref{160420_05})]
By item (\ref{160420_04}) we know that $a \to b \leq \alpha$. Then $\alpha^+ \leq (a \to b)^+$. Therefore $1 = \alpha \vee \alpha^+ \leq \alpha \vee (a \to b)^+$.

\item[(\ref{160420_06})]
\noindent $(a \to b)^+  \vee \beta \vee \gamma $
$\overset{ def\  of\  \beta 
}{=}  (a \to b)^+  \vee (a \vee a^*)^+ \vee \gamma $
$\overset{   \ref{lemma_properties_RDPCH} (\ref{140420_02}) 
}{=}  [(a \to b) \wedge (a \vee a^*)]^+  \vee \gamma $
$\overset{  
}{=}  [\{(a \to b) \wedge a\} \vee \{(a \to b) \wedge a^*\}]^+  \vee \gamma $
$\overset{  (H\ref{conditionH2}) 
}{=}  ((b \wedge a) \vee ((a \to b) \wedge a^*))^+ \vee \gamma $
$\overset{   \ref{lemma_properties_H} (\ref{140420_05}) 
}{=}  ((b \wedge a) \vee a^*)^+ \vee \gamma $
$\overset{   \ref{lemma_properties_RDPCH} (\ref{140420_03}) 
}{=}  ((b \wedge a) \vee a^* \vee \gamma)^+ \vee \gamma $
$\overset{ def\  of\  \gamma 
}{=}  ((b \wedge a) \vee a^* \vee a^* \vee b \vee b^*)^+ \vee \gamma $
$\overset{  
}{=} ( a^* \vee b \vee b^*)^+ \vee \gamma $
$\overset{  def\  of\  \gamma 
}{=}  \gamma^+ \vee \gamma $
$\overset{  
}{=}  1 $.

\item[(\ref{160420_07})]
In view of items (\ref{160420_05}) and (\ref{160420_06}) we know that $(a \to b)^+ \vee [\alpha \wedge (\beta \vee \gamma)] = [\alpha \vee (a \to b)^+] \wedge [(a \to b)^+  \vee \beta \vee \gamma] = 1$. Hence $(a \to b)^+ \geq [\alpha \wedge (\beta \vee \gamma)]^+$.

\item[(\ref{160420_08})] By Lemma \ref{lemma_properties_H} (\ref{140420_05}), we have $a^* \leq a \to b$. Hence $a^* \wedge (a \to b)^* \leq (a \to b) \wedge (a \to b)^* = 0$ and, consequently, $(a \to b)^* \leq a^{**}$. Similarly, since $b \leq a \to b$ by Lemma \ref{lemma_properties_H} (\ref{140420_07}), we get $b \wedge (a \to b)^* = 0$. Therefore $(a \to b)^* \leq b^*$. Hence,
\begin{equation} \label{equation_160420_03}
(a \to b)^*  \leq a^{**} \wedge b^*.
\end{equation}
Since
\noindent $(a \to b)^* \wedge \alpha $
$\overset{ def\  of\  \alpha 
}{=}  (a \to b)^* \wedge (a^* \vee b^{**})^{**} $
$\overset{  (\ref{equation_160420_01}) 
}{=}  (a \to b)^* \wedge  (a^{**} \wedge b^{*})^{*} $
$\overset{   \ref{lemma_properties_RDPCH} (\ref{140420_11}) 
}{=}  (a \to b)^* \wedge  [(a \to b)^* \wedge a^{**} \wedge b^{*}]^{*} $
$\overset{  (\ref{equation_160420_03}) 
}{=}  (a \to b)^* \wedge (a \to b)^{**}  $
$\overset{  
}{=}  0, $
we have that $(a \to b)^* \wedge [\alpha \wedge (\beta \vee \gamma)]  \leq (a \to b)^* \wedge \alpha = 0$. Therefore 
$(a \to b)^* \leq (\alpha \wedge (\beta \vee \gamma))^*$.

\item[(\ref{160420_09})]
From
\noindent $(a \to b) \wedge (\alpha \wedge \gamma)^* $
$\overset{  
}{=}  (a \to b) \wedge ((\alpha \wedge \gamma) \to 0) $
$\overset{  (H\ref{conditionH3}) 
}{=}  (a \to b) \wedge [\{(a \to b) \wedge \alpha \wedge \gamma\} \to \{(a \to b) \wedge 0\}] $
$\overset{ 
	(\ref{160420_04}) 
}{=}  (a \to b) \wedge [\{(a \to b) \wedge \gamma\} \to \{(a \to b) \wedge 0\}] $
$\overset{  (H\ref{conditionH3}) 
}{=}  (a \to b) \wedge \gamma^* $
$\overset{ def\  of\  \gamma 
}{=}  (a \to b) \wedge (a^* \vee b \vee b^*)^* $
$\overset{   \ref{lemma_properties_RDPCH} (\ref{140420_09}) 
}{=}  (a \to b) \wedge (a^{**} \wedge b^* \wedge b^{**}) $
$\overset{  
}{=}  0, $
we conclude that
\begin{equation} \label{equation_160420_04}
(a \to b) \wedge (\alpha \wedge \gamma)^* = 0.
\end{equation}
Also, in view of Lemma \ref{lemma_properties_RDPCH} (\ref{140420_09}), we have
\noindent $(a \to b) \wedge [\alpha \wedge (\beta \vee \gamma)]^* $
$\overset{  
}{=}  (a \to b) \wedge [(\alpha \wedge \beta) \vee (\alpha \wedge \gamma)]^* $
$\overset{  (\ref{equation_160420_04}).
}{=}  0 $. 
\noindent Consequently, $[\alpha \wedge (\beta \vee \gamma)]^* \leq (a \to b)^*$.
\end{itemize}	
\end{proof}

We are ready to present our main theorem of this paper.

\begin{Theorem} \label{Main}
\begin{thlist} 
	\item[a] Let $\mathbf A  := \langle A, \vee, \wedge,  \rightarrow, ^+, 0, 1 \rangle \in \mathbb{RDPCH}$.  Then 	$\mathbf A \models x \to y \approx \kappa(x,y)$.
	\item[b] Let $\mathbf A  := \langle A, \wedge, \vee,  ^*, \leftarrow,  1, 0 \rangle \in  \mathbb{RPCH}^d$.  Then 	$\mathbf A \models x \leftarrow y \approx \kappa^d (x,y)$.
	\item[c]
 Let $\mathbf A := \langle A, \vee, \wedge, \to, \leftarrow, 0, 1 \rangle 
\in  \mathbb{RDBLH}$. Then
\begin{thlist} 
\item[i] $\mathbf A \models x \to y \approx \kappa(x,y)$,
\item[ii] $\mathbf A \models x \leftarrow y \approx \kappa^d (x,y)$. 
\end{thlist}

\end{thlist}
\end{Theorem}

\begin{proof}
The identity $(x \to y)^*  \approx ((x^* \lor y^{**})^{**}  \land  ((x \lor x^*)' \lor (x^* \lor (y \lor y^*))))^*$ is valid in $\mathbf A$, in view of items (\ref{160420_08}) and (\ref{160420_09})  of Lemma \ref{lemma_properties_alphabetagamma}.
Also, by items (\ref{140420_01}) and (\ref{160420_07}) of Lemma \ref{lemma_properties_alphabetagamma}  we can easily verify that the identity $(x \to y)^+  \approx  ((x^* \lor y^{**})^{**}  \land  ((x \lor x^*)' \lor (x^* \lor (y \lor y^*))))^+$ is true in $\mathbf A$.	
Therefore, by (R), we conclude that the identity $x \to y  \approx (x^* \lor y^{**})^{**}  \land  ((x \lor x^*)' \lor (x^* \lor (y \lor y^*)))$ holds in $\mathbf A$, thus proving (a).   A dual argument will prove (b), while (c) follows from (a) and (b).   
\end{proof}

We now give some applications of Theorem \ref{K_Theorem} and Theorem \ref{Main}. 


\begin{Corollary} \label{CorA}
	$\mathbb{RDBLP}$ and $\mathbb{RDBLH}$ are term-equivalent to each other.  More explicitly,
	\begin{enumerate}
		\item[{\rm(a)}] For $\mathbf A = \langle A, \lor, \land, {^*}^{\mathbf{A}}, {^+}^{\mathbf{A}}, 0, 1 \rangle    \in \mathbb{RDBLP}$, let $\mathbf{A}^{dblh} := \langle A, \lor, \land, \to, \leftarrow, 0, 1 \rangle$, where $\to  \ := \kappa(x,y)$ and $\leftarrow \ := \kappa^d (x,y)$. 
		Then $\mathbf{A}^{dblh} \in \mathbb{RDBLH}$.
		
		\item[{\rm(b)}] For $\mathbf A = \langle A, \lor, \land, \to^{\mathbf{A}}, \leftarrow^{\mathbf{A}}, 0, 1 \rangle \in \mathbb{RDBLH}$, let $\mathbf{A}^{dblp}:= \langle A, \lor, \land, ^*, ^+, 0, 1 \rangle$, where ${^*}^\mathbf{A}$ is defined by $x{^*}^\mathbf{A} := x \to^{\mathbf A} 0 $ and ${^+}^\mathbf{A} $ is defined by:\\
		 $x{^+}^{\mathbf A} := x \leftarrow^{\mathbf A} 1 $.     
		Then $\mathbf{A}^{dblp} \in \mathbb{RDBLP}$.
		
		\item[{\rm(c)}] If $\mathbf A \in \mathbb{RDBLP}$, then $\mathbb(A^{{dblh}})^{dblp}  = \mathbf A$.
		\item[{\rm(d)}]  If $\mathbf A \in \mathbb{RDBLH}$, then $\mathbf(A^{dblp})^{dblh}  = \mathbf A$.
	\end{enumerate}
\end{Corollary}

\begin{proof}
	(a) follows from Katrinak's theorem \ref{K_Theorem}, while (b) is well known.  \\
 (c): Let $\mathbf{A} := \langle A, \lor, \land, {^*}^{\mathbf{A}}, {^+}^{\mathbf{A}}, 0, 1 \rangle \in \mathbb{RDBLP}$. Let $\mathbf{A_1} :=\mathbf{A}^{dblh} := \langle A, \lor, \land, \kappa, \kappa^d, 0, 1 \rangle$.  Then $\mathbf{A}_1 \in \mathbb{RDBLH}$. by (a).  Now $\mathbf{A}_1^{dblp}:= \langle A, \lor, \land, {^*}^{\mathbf{A_1}}, {^+}^{\mathbf{A_1}}, 0, 1 \rangle$, where 
 $x{^*}^{\mathbf{A_1}}:= \kappa(x,0)$ and 
 ${x^+}^{\mathbf{A_1}} := \kappa^d (x,0)$.	  Observe that $\kappa(x,0)=x^*$ and  $\kappa^d (x,1)=x^+$.  Then it follows that
  ${^*}^{\mathbf{A_1}}=^*$ and ${^+}^{\mathbf{A_1}}= ^+$, implying $\mathbf{A_1}= \mathbf{A}$. \\
 (d): Let $\mathbf{A} := \langle A, \lor, \land, \to^{\mathbf{A}}, \leftarrow^{\mathbf{A}}, 0, 1 \rangle  \in \mathbb{RDBLH}$. Let $\mathbf{A_2} :=\mathbf{A}^{dblp} := \langle A, \lor, \land, {^*}, \kappa^d, 0, 1 \rangle$.  Then $\mathbf{A}_2 \in \mathbb{RDBLP}$. by (b).  Now $\mathbf{A}_2^{dblh}:= \langle A, \lor, \land, {\kappa}^{\mathbf{A_2}}, {\kappa^d}^{\mathbf{A_2}}, 0, 1 \rangle$, where 
 ${\kappa}^{\mathbf{A_2}}$ and ${\kappa^d}^{\mathbf{A_2}}$ are as defined earlier.
   Observe that $\kappa^{\mathbf A_2} = \  \to^{\mathbf A_2}$ and  $\kappa{^d}^{\mathbf{A_2}} = \  \leftarrow^{\mathbf{A_2}}$ by Theorem \ref{Main}.  Hence, it follows that
$\mathbf{A_2}= \mathbf{A}$. 	
\end{proof}

\begin{Corollary} \label{Cor_termeq1}
$\mathbb{RDBLP}$ and $\mathbb{RDPCH}$ 
are term-equivalent to each other.  More explicitly,
\begin{enumerate}
   \item[{\rm(a)}] For $\mathbf A \in \mathbb{RDBLP}$, let $\mathbf{A}^{dpch} := \langle A, \lor, \land, \to, ^+, 0, 1 \rangle$, where $\to$ is as defined in Theorem \ref{K_Theorem}. 
   Then $\mathbf{A}^{dpch} \in \mathbb{RDPCH}$.
   
   \item[{\rm(b)}] For $\mathbf A \in \mathbb{RDPCH}$, let $\mathbf{A}^{dblp} := \langle A, \lor, \land, ^*, ^+, 0, 1 \rangle$, where $^*$ is defined by: $x^* := x \to 0 $.   
   Then $\mathbf{A}^{dblp} \in \mathbb{RDBLP}$.

   \item[{\rm(c)}] If $\mathbf A \in \mathbb{RDBLP}$, then $\mathbf(A^{{dpch}})^{dblp}  = \mathbf A$.
   \item[{\rm(d)}]  If $\mathbf A \in \mathbb{RDPCH}$, then $\mathbb(A^{dblp})^{dpch} = \mathbf A$.
\end{enumerate}
\end{Corollary}

\begin{proof}
(a) follows from Katrinak's theorem \ref{K_Theorem}, while (b) is well-known.  The verification of (c) and (d), being similar to that of (c) and (d) of Corollary \ref{CorA}, is left to the reader. 
\end{proof}

Similarly, the following corollary is also proved.
\begin{Corollary} \label{Cor_termeq1}
$\mathbb{RDBLP}$ and $\mathbb{RPCH}^d$ 
are term-equivalent to each other. 
\end{Corollary}

The following corollary is immediate from the preceding corollaries.
\begin{Corollary}
The varieties $\mathbb{RDBLP}$, $\mathbb{RDPCH}$,   $\mathbb{RPCH}^d$, and $\mathbb{RDBLH}$,                
are term-equivalent to one  another.
\end{Corollary}

The following corollary is immediate from the preceding corollary.  Let $\mathcal{L}_{\rm V}(\mathbb{V})$ denote the lattice of subvarieties of the variety $\mathbb{V}$ of algebras. 

\begin{Corollary} \label{C2.7}
$\mathcal{L}_{\rm V}(\mathbb{RDBLP}) \cong \mathcal{L}_{\rm V}(\mathbb{RDPCH}) \cong \mathcal{L}_{\rm V}(\mathbb{RPCH}^d) \cong \mathcal{L}_{\rm V}(\mathbb{RDBLH})$. 
\end{Corollary}

\medskip
\subsection{Regular De Morgan p-algebras, Regular De Morgan Heyting algebras, Regular De Morgan double Heyting algebras and  Regular De Morgan double $p$-algebras}  \label{section_RDMP}

\

\medskip
We now give more consequences of Theorem \ref{K_Theorem} and Theorem \ref{Main}. 
The proofs of the following corollaries are similar to those of the previous corollaries.

\begin{Corollary} \label{RDMP_Theorem}
$\mathbb{RDMP}$ and $\mathbb{RDMH}$ 
are term-equivalent to each other.  More explicitly,
\begin{enumerate}
   \item[{\rm(a)}] For $\mathbf A = \langle A, \lor, \land, ^*, ', 0, 1 \rangle \in \mathbb{RDMP}$, let $\mathbf{A}^{dmh} := \langle A, \lor, \land, \to, ', 0, 1 \rangle$, where $\to$ is defined by:
   $x \to y := (x^{\ast} \lor y^{\ast \ast}){^{\ast \ast}} \land [(x \lor x^{\ast})'{^*}'  \lor x^{\ast} \lor y \lor y^{\ast}].$
   Then $\mathbf{A}^{dmh} \in \mathbb{RDMH}$.
   
   \item[{\rm(b)}] For $\mathbf A = \langle A, \lor, \land, \to, ', 0, 1 \rangle \in \mathbb{RDMH}$, let $\mathbf{A}^{dmp} := \langle A, \lor, \land, ^*, ', 0, 1 \rangle$, where $^*$ is defined by $x^* := x \to 0 $.     
   Then $\mathbf{A}^{dmp} \in \mathbb{RDMP}$.

   \item[{\rm(c)}] If $\mathbf A \in \mathbb{RDMP}$, then $(\mathbf{A}^{dmh})^{dmp}  = \mathbf A$.
   \item[{\rm(d)}]  If $\mathbf A \in \mathbb{RDMH}$, then $(\mathbf{A}^{dmp})^{dmh}  = \mathbf A$.
\end{enumerate}
\end{Corollary}

\begin{Corollary} 
$\mathbb{RDMP}$ and $\mathbb{RDMDBLH}$ are term-equivalent to each other.  More explicitly,
\begin{enumerate}
   \item[{\rm(a)}] For $\mathbf A \in \mathbb{RDMP}$, let $\mathbf{A}^{dmdblh} := \langle A, \lor, \land, \to, \leftarrow, ', 0, 1 \rangle$, where $\to$ and $\leftarrow$ are defined earlier. 
   Then $\mathbf{A}^{dmdblh} \in \mathbb{RDMDBLH}$.
   
   \item[{\rm(b)}]  For $\mathbf A \in \mathbb{RDMDBLH}$, let $\mathbf{A}^{dmp} := \langle A, \lor, \land, \to, ' , 0, 1 \rangle$, 
   Then $\mathbf{A}^{dmp} \in \mathbb{RDMP}$.

   \item[{\rm(c)}] If $\mathbf A \in \mathbb{RDMP}$, then $\mathbb(A^{{dmdblh}})^{dmp}  = \mathbf A$.
   \item[{\rm(d)}]  If $\mathbf A \in \mathbb{RDMDBLH}$, then $\mathbb(A^{dmp})^{dmdblh}  = \mathbf A$.
\end{enumerate}
\end{Corollary}

\begin{Corollary} 
$\mathbb{RDMP}$ and $\mathbb{RDMDBLP}$ are term-equivalent to each other. 
\end{Corollary}

\begin{Corollary} \label{Cor3A}
$\mathcal{L}_{\rm V}(\mathbb{RDMP}) \cong \mathcal{L}_{\rm V}(\mathbb{RDMH}) \cong \mathcal{L}_{\rm V}(\mathbb{RDMDBLH}) \\
\cong \mathcal{L}_{\rm V}(\mathbb{RDMDBLP})$.
\end{Corollary}

 Using Corollary \ref{Cor3A} and the results from \cite{AdSaVc19}  and \cite{ AdSaVc20} we can conclude the following.

\begin{Corollary} \label{C1}
$|\mathcal{L}_{\rm V}(\mathbb{RDMP})| = |\mathcal{L}_{\rm V}(\mathbb{RDMH})| = |\mathcal{L}_{\rm V}(\mathbb{RDMDBLH})| $ \\
$= |\mathcal{L}_{\rm V}(\mathbb{RDMDBLP})| = 2^{\aleph_0}$.
\end{Corollary}




\vspace{.5cm}
\section{Logically speaking} 

In this section our goal is to introduce new logics using the results proved in \cite{CoSa22} and also show that these logics are algebraizable having the varieties considered above as their algebraic semantics.  
%
%
To achieve this goal, we will first recall some preliminaries of the Abstract Algebraic Logic \cite{BlPi89, Fo16} and also certain definitions and results from \cite{Sa12, CoSa22}. 

An algebra $\mathbf A = \langle A, \lor,  \wedge,  \to, ', 0,1\rangle$ is a  
dually hemimorphic semi-Heyting algebra  (\cite{Sa12})  if  $\mathbf A$ satisfies the following conditions:

(a)  $\langle A,\lor, \wedge,  \to, 0,1\rangle $ is a semi-Heyting algebra (defined in Section 1),

(b) $0' \approx 1$,  

(c) $1' \approx 0$, 

(d) $(x \land y)' \approx x' \lor y'$  {\rm ($\land$-De Morgan law)}. 

The variety of dually hemimorphic semi-Heyting algebras will be denoted by $\mathbb{DHMSH}$.



\medskip
We now present the basic definitions and results of Abstract Algebraic Logic that will be useful
 later in this section. \\ 


{\bf Languages, Formulas and Logics}

\medskip

\indent A language $\mathbf L$ is a set of  finitary operations (or connectives), each with a fixed arity $n \geq 0$.   In this paper, we identify $\bot$ and $\top$ with $0$ and $1$ respectively and thus consider the languages $\langle \lor, \land, \to, \sim, 
 \bot, \top\rangle$ and $\langle \lor, \land, \to, ', 0, 1 \rangle$ as the same.  
For a countably infinite set $\mbox{\it Var}$ of propositional variables, the {\it formulas} of the language $\mathbf L$ are inductively defined as usual.  The set of formulas in the language {\bf L} will be denoted by $\mbox{\it Fm}_{\,\bf L}$

The set of formulas $\mbox{\it Fm}_{\,\bf L}$  can be turned into an algebra of formulas, denoted by ${\bf Fm}_{\bf L}$, in the usual way. Throughout the paper, $\Gamma$ denotes a set of formulas and lower case Greek letters denote formulas.  The homomorphisms from the formula algebra ${\bf Fm}_{\bf L}$ into an $\mathbf L$-algebra (i.e, an algebra of type $\mathbf L$) $\mathbf A$ are called {\em interpretations} (or {\em valuations}) in A. The set of all such interpretations is denoted by  $Hom({\bf Fm}_{\bf L},\mathbf A)$. If $h \in Hom({\bf Fm}_{\bf L},\mathbf A)$ then the {\em interpretation of a formula} $\alpha$ under $h$ is its image $h \alpha \in A$, while  $h\Gamma$ denotes the set $\{h\phi \ | \ \phi \in \Gamma\}$.  \\

{\bf Consequence Relations}: 

A {\it consequence relation} on $Fm_{\mathbf L}$ is a binary relation $\vdash$ between sets of
formulas and formulas that satisfies the following conditions for all $\Gamma$, $\Delta  \subseteq Fm_{\mathbf L}$ and $\phi \in Fm_{\bf L}$ :

(i) $\phi \in \Gamma$ implies $\Gamma \vdash \phi,$

(ii) $ \Gamma \vdash \phi$ and $\Gamma \subseteq \Delta$ imply $\Delta \vdash \phi,$

(iii) $\Gamma \vdash \phi$ and $\Delta \vdash \beta$ for every $\beta \in \Gamma$ imply $\Delta \vdash \phi.$

A consequence relation $\vdash$ is {\it finitary} if
$\Gamma \vdash \phi$ implies $\Gamma' \vdash \phi$ for some finite $\Gamma' \subseteq \Gamma$. 



\medskip
 {\bf Structural Consequence Relations}: 

A consequence relation $\vdash$ is {\it structural} if

$\Gamma \vdash \phi$  implies $\sigma(\Gamma) \vdash \sigma(\phi)$ for every substitution $\sigma$, where $\sigma(\Gamma) :=\{\sigma\alpha : \alpha \in \Gamma\}$.

\medskip
 {\bf Logics}:

A {\bf logic} (or {\bf deductive system}) is a pair $\mathcal S := \langle \mathbf{L}, \vdash_{\mathcal S} \rangle$, where $\mathbf{L}$ is a propositional
language and $\vdash_{\mathcal S}$ is a finitary and structural consequence relation on Fm$_{\mathbf L}$.


%
A {\it rule of inference} is a pair $\langle\Gamma, \phi \rangle$, 
where $\Gamma$ is a finite set of formulas (the premises of the rule) and $\phi$ is a
formula. 

One way to present a logic $\mathcal S$ is by displaying it (syntactically) in {\bf Hilbert-style}; that is, giving its axioms and rules of inference which induce a
 consequence relation $\vdash_S$ as follows:

 $\Gamma \vdash_S \phi$ if there is a a {\bf proof} (or, a {\bf derivation}) of $\phi$ from $\Gamma$, where a proof is defined as a 
 sequence of formulas $\phi_1, \dots, \phi_n$, $n \in \mathbb N$, such that
$\phi_n = \phi$, and for every $i \leq n$, one of the following conditions holds:

(i) $\phi_i  \in \Gamma$,

(ii) there is an axiom $\psi$  and a substitution $\sigma$ such that $\phi_i = \sigma \psi $, 

(iii) there is a rule $\langle \Delta, \psi \rangle$ and a substitution $\sigma$  such that $\phi_i  = \sigma\psi$  and
$\sigma(\Delta) \subseteq \{\phi_j : j < i\}$.    

\medskip
{\bf Equational Consequence }

Let $\mathbf L$ denote a language. 
 Identities in $\mathbf L$ are ordered pairs of $\mathbf L$-formulas that will be written in the form $\alpha \approx \beta$.
An interpretation $h$ in $\mathbf A$ satisfies an identity $\alpha \approx \beta$ if $h \alpha = h \beta$.  
We denote this satisfaction relation by the notation: $\mathbf{A} \models_h
\alpha \approx \beta $. An algebra $\mathbf A$ {\it satisfies the equation} $\alpha \approx \beta$ if all the interpretations in $\mathbf A$ satisfy it; in symbols,
$$\mathbf{A} \models \alpha \approx \beta \mbox{ if and only if } \mathbf{A} \models_h
\alpha \approx \beta, \mbox{ for all } h \in Hom({\bf Fm}_{\bf L},\mathbf A).$$
A class $\mathbb K$ of algebras  {\it satisfies the identity} $\alpha \approx \beta$ when all the algebras in $\mathbb K$ satisfy it; i.e.
$$\mathbb{K} \models \alpha \approx \beta \mbox{ if and only if } \mathbf{A}\models \alpha \approx \beta, \mbox{ for all }\mathbf A \in \mathbb K.$$ 

If $\bar{x}$ is a sequence of variables and $h$ is an interpretation in $\mathbf{A}$, then we write $\bar{a}$ for $h(\bar{x})$.
For a class $\mathbb K$ of $\mathbf L$-algebras, 
we define the relation $\models_{\mathbb K}$ that holds between a set $\Delta$ of identities and a single identity $\alpha \approx  \beta$ as follows: \\ 
$$\Delta \models_{\mathbb K} \alpha \approx \beta \text{ if and only if }$$

for every $\mathbf A \in \mathbb{K}$ and every interpretation $\bar{a}$ of the variables of $\Delta \cup  \{\alpha \approx \beta\}$ in $\mathbf{A}$, 

If $\phi^{\mathbf A}(\bar{a} ) = \psi^{\mathbf A} (\bar{a}) $,  for every $\phi \approx \psi \in \Delta$, then \ 
$\alpha^{\mathbf A}(\bar{a}) = \beta^{\mathbf A}(\bar{a})$.\\

In this case, we say that $\alpha \approx \beta$ is a $\mathbb K$-consequence of $\Delta$. 
The relation $\models_K$ is called the {\it semantic equational consequence relation} determined by K. \\

{\bf Algebraic Semantics} 

Let $\langle {\bf L}, \vdash_{\bf L} \rangle$ be a finitary logic (i.e., deductive system) and $\mathbb{K}$ a class of {\bf L-algebras}. $\mathbb{K}$ is called an ``algebraic semantics'' for $\langle {\bf L}, \vdash_{\bf L} \rangle$ 
if $\vdash_{\bf L}$ can be interpreted in $\vdash_\mathbb{K}$ in the following sense:

There exists a finite set $\delta_i (p) \approx \epsilon_i (p)$, for $i  \leq n$, of identities with a single variable $p$ such that, for all $\Gamma \cup \phi \subseteq Fm$, 
$$ \text{(A)} \quad \Gamma \vdash_{\bf L} \phi \Leftrightarrow  \{\delta_i [\psi /p]  \approx \epsilon_i[\psi /p], i  \leq n, \psi \in \Gamma\} 
\models_K \delta_i [\phi /p] \approx \epsilon_i [\phi /p],$$ 

where $\delta[\psi / p]$  denotes the formula obtained by the substitution of $\psi$ at every occurrence of $p$ in $\delta. $

 The identities $\delta_i \approx \epsilon_i$, for $i \leq n$, are called ``defining identities'' for $\langle L, \vdash_L\rangle$ and $\mathbb{K}$. 

\newpage
{\bf Equivalent Algebraic Semantics and Algebraizable Logic}

\medskip 
Let $\mathcal S$ be a logic over a language {\bf L} and $\mathbb K$ an algebraic semantics
of S with defining equations $\delta_i(p) \approx \epsilon_i(p)$, $i \leq n$.  Then, $\mathbb K$ is an {\bf equivalent
algebraic semantics} of $\mathbf{S}$ if there exists a finite set $\{\Delta_j (p, q) : j \leq m\}$ of
formulas in two variables satisfying the condition:\\

For every  $\phi  \approx \psi \in Eq_{\mathbf L}$ and $j \leq m$, 
$$\phi  \approx \psi 
           \ \models_K \{\delta_i(\Delta_j (\phi, \psi)) \approx \epsilon_i(\Delta_j (\phi, \psi)) : i \leq n, j \leq m\}$$            
\indent and
$$\{\delta_i(\Delta_j (\phi, \psi)) \approx \epsilon_i(\Delta_j (\phi, \psi)) : i \leq n, j \leq m\} \ \models_K  \ \phi  \approx \psi. $$\\
The set $\{\Delta_j (p, q) : j \leq m \}$ is called an {\bf equivalence system}. \\

A logic is {\bf BP-algebraizable (in the sense of Blok and Pigozzi)} if it has an equivalent algebraic semantics.\\

{\bf Axiomatic Extensions of Algebraizable logics} 

\medskip  
A logic $\mathbf{\mathit{S'}}$ is an {\it axiomatic extension} of $\mathbf{\mathit{S}}$ if $\mathbf{\mathit{S'}}$ is obtained by adjoining new axioms but keeping the rules of inference the same as in $\mathbf{\mathit{S}}$. 
Let $Ext(\mathbf{\mathit{S}})$ denote the lattice of axiomatic extensions of a logic $\mathbf{\mathit{S}}$ and 
$\mathbf{L_V(\mathbb K)}$ denote the lattice of subvarieties of a variety $\mathbb K$ of algebras. 

The following important theorems, due to Blok and Pigozzi, were first proved in 
\cite{BlPi89}.  
	
\begin{Theorem} \cite{BlPi89} \label{duallyIso}     
Let $\mathbf{\mathit{S}}$ be a BP-algebraizable logic whose equivalent algebraic semantics $\mathbb K$ is a variety.  Then
$Ext(\mathbf{\mathit{S}})$ is dually isomorphic to $\mathbf{L_V(\mathbb K)}$. \\
\end{Theorem} 

\begin{Theorem} \cite{BlPi89} \label{Ext-Alg}
Let $\mathbf{\mathit{S}}$ be a BP-algebraizable logic and $\mathbf{\mathit{S'}} $  be an axiomatic extension of $\mathbf{\mathit{S}}$.  Then  $Ext(\mathbf{\mathit{S'}})$ is also BP-algebraizable.
\end{Theorem}

\subsection{The Dually Hemimorphic Semi-Heyting Logic}

\

\medskip
\indent  The new logics that we intend to present are going to be axiomatic extensions of the logic  called ``dually hemimorphic semi-Heyting logic'' 
($\mathcal{DHMH}$, for short) which is introduced and 
 shown, in \cite{CoSa22}, to be BP-algebraizable 
 with the variety $\mathbb{DHMSH}$ of dually hemimorphic semi-Heyting algebras as its equivalent algebraic semantics.  Therefore, we first describe the logic $\mathcal{DHMH}$,

 
Following (\cite{CoSa22}), the {\it dually hemimorphic semi-Heyting logic} ($\mathcal{DHMSH}$),  is defined in the language 
$\langle \vee, \wedge,   \to, \sim,  \bot, \top \rangle$ and has the following axioms and rules of inference, where $\to_H$  is defined by 
$\alpha \to_H \beta := \alpha \to (\alpha \land \beta)$ and \\
$\alpha \leftrightarrow_H \beta :=: (\alpha \to_H \beta) \land (\beta \to_H \alpha)$: \\  

{\bf \noindent AXIOMS:}  
\begin{thlist}
\item[1] \noindent \rm $\alpha \to_H (\alpha \vee \beta)$, \label{axioma_supremo_izq}  

\item[2] \noindent $\beta \to_H (\alpha \vee \beta)$, \label{axioma_supremo_der}  

\item[3] \noindent $(\alpha \to_H  \gamma) \to_H [(\beta \to_H  \gamma) \to_H \{(\alpha \vee \beta) \to_H \gamma\}]$, \label{axioma_supremo_cota_inferior} 

\item[4] \noindent $(\alpha \wedge \beta) \to_H {\alpha}$, \label{axioma_infimo_izq} 

\item[5] \noindent $(\gamma \to_H  \alpha) \to_H  [\gamma \to_H  \beta) \to_H  (\gamma \to_H  (\alpha \wedge \beta))]$, \label{axioma_infimo_cota_superior} 

\item[6] \noindent $\top$, \label{axioma_top} 

\item[7] \noindent $\bot \to_H \alpha$, \label{axioma_bot} 

\item[8] \noindent $[(\alpha \land \beta) \to_H  \gamma] \to_H  [\alpha \to_H  (\beta \to_H  \gamma)]$, \label{axioma_condicRes_InfAImplic}  

\item[9] \noindent $[\alpha \to_H  (\beta \to_H  \gamma )] \to_H  [(\alpha \wedge \beta) \to_H  \gamma]$, \label{axioma_condicRes_ImplicAInf}  

\item[10] \noindent $(\alpha \to_H  \beta) \to_H  [(\beta \to_H  \alpha) \to_H ((\alpha \to \gamma) \to_H  (\beta \to \gamma))]$, \label{axioma_BuenaDefImplic2} \label{axioma_ImplicADerecha}  

\item[11] \noindent $(\alpha \to_H \beta) \to_H  [(\beta \to_H  \alpha) \to_H ((\gamma \to \beta) \to_H  (\gamma \to \alpha))]$, \label{axioma_BuenaDefImplic1} \label{axioma_ImplicAIzquierda}  


\item[12]  $\top \to_H  \ \sim\bot$, \label{axiom_bot_neg} 

\item[13] $\sim\top \to_H  \bot$, \label{axiom_top_neg} 

\item[14]  $\sim(\alpha \wedge \beta) \to_H  (\sim\alpha \ \vee \sim\beta)$. \label{axiom_DM_law1} \\ 


\end{thlist}

{\bf \noindent RULES OF INFERENCE:}\\

	       {\bf (MP)}  \quad From $\phi$ and $\phi \to_H  {\gamma}$, deduce $\gamma$  ({\rm semi-Modus Ponens}),\\
 
          {\bf (CP)} \quad From  $\phi \to_H  \gamma$, deduce $\sim\gamma \to_H  \sim\phi$ \rm(semi-Contraposition Rule\rm).\\

 The following theorem is proved in \cite[Corollary 5.4]{CoSa22}.
\begin{Theorem} \label{CorExt} 
 The logic $\mathcal{DHMSH}$ is BP-algebraizable, and  the variety  
 $\mathbb{DHMSH}$ is the equivalent algebraic semantics for $\mathcal{DHMSH}$ with the defining identity $p \approx p \to_H p$ \rm(equivalently, $p \approx 1$) and the equivalence formulas $\Delta = \{p \to_H q, q \to_H p \}$.
 \end{Theorem}

The following theorem is immediate from Theorem \ref{duallyIso} and Theorem \ref{CorExt}.
\begin{Theorem} 
The lattice $Ext(\mathbf{\mathit{DHMH}})$ of axiomatic extensions of $\mathbf{\mathit{DHMSH}}$ is dually isomorphic to the lattice 
$\mathbf{L_V(\mathbb{DHMSH})}$ of subvarieties of the variety $\mathbb{DHMH}$.
\end{Theorem}

The following theorem, which is an immediate consequence of Theorem \ref{DHMSHduallyIso}, 
is crucial in the rest of this section.

 \begin{Theorem} \cite[Theorem 5.9]{CoSa22} \label{T5.9} 
	For every axiomatic extension $\mathcal E$ of the logic $\mathcal{DHMSH}$,           
	$\mathbf{Mod}(\mathcal E)$ is an equivalent algebraic semantics of $\mathcal E$, where $\mathbf{Mod}(\mathcal E) := \{\mathbf{A} \in \mathbb{DHMSH}: \mathbf{A} \models \delta \approx 1, \mbox{ for every } \delta \in \mathcal E\} $. 
	
 \end{Theorem}


We are ready to present a new logic called $\mathcal{DPCH}$.		
\begin{definition}
The logic $\mathcal{DPCH}$ is defined in \cite{CoSa22}, as an axiomatic extension of the logic $\mathcal{DHMSH}$, by
the addition of 
the following axioms:
\begin{thlist}
\item[15] $(\alpha \land \beta) \to \alpha,$ 

\item[16] $\sim \bot \to_H  \top$,  


\item[17] $\sim \sim (\alpha \lor \beta) \  \leftrightarrow_H \    (\sim \sim \alpha \ \lor \ \sim \sim \beta), $

\item[18] $ (\sim \sim \alpha \lor \alpha) \ \leftrightarrow_H \   \alpha,$

\item[19] $\alpha \ \lor \sim \alpha$ (equivalently, $\sim \alpha \ \lor \sim \sim \alpha$).
\end{thlist}
\end{definition}

The following theorem is proved in \cite{CoSa22}.
\begin{Theorem} {\rm(\cite{CoSa22})} \label{T_DPCH}
The logic $\mathcal{DPCH}$ is BP-algebraizable with 
   $\mathbb{DPCH}$ 
   as its equivalent algebraic semantics. 
\end{Theorem}                   

We now define the logic $\mathcal{RDPCH}$.
\begin{definition}
Let $\mathcal{RDPCH}$  be defined as the axiomatic extension of the logic $\mathcal{DPCH}$ 
by the axiom:
\begin{equation} \tag{$\mathcal{ M}_1$}
  ((\alpha \wedge \sim \alpha) \vee (\beta \vee \beta^{\ast})) \leftrightarrow_H (\beta \vee \beta^{\ast}), 
    \end{equation}
  where $\beta^{\ast} = \beta \to 0$.

 \end{definition}
 
 The following corollary is immediate from Theorem \ref{T5.9}  and Theorem \ref{T_DPCH}.
 
 \begin{Corollary}  \label{alg_semantics1} 
The logic $\mathcal{RDPCH}$ is BP-algebraizable 
with the variety
   $\mathbb{RDPCH}$ 
 as its equivalent algebraic semantics.  
\end{Corollary}

 The logic $\mathcal{RPCH}^d$ is defined by dualizing the axioms and inference rules of the logic $\mathcal{RDPCH}$.  
 
 Hence we get the following corollary.
 
 \begin{Corollary}  \label{alg_semantics2}
The logic $\mathcal{RPCH}^d$ is BP-algebraizable with  
 the variety
   $\mathbb{RPCH}^d$ 
 as its equivalent algebraic semantics.  
\end{Corollary}  

\begin{definition}  
  The logic $\mathcal{DMH}$ is defined in \cite{CoSa22} as an axiomatic extension of the logic $\mathcal{DHMH}$, by
the addition of the following axioms:  (15) and  
\begin{thlist}
\item[20] $\sim \sim \alpha \leftrightarrow_H \alpha$.  
\end{thlist}
\end{definition}

The following theorem is also proved in \cite{CoSa22}.

\begin{Theorem} {\rm(\cite{CoSa22})} \label{T_DMH}
the logic $\mathcal{DMH}$ is BP-algebraizable with 
the variety $\mathbb{DMH}$ 
   as its equivalent algebraic semantics.  
\end{Theorem}   

We now define the logic $\mathcal{RDMH}$.
\begin{definition}
Let $\mathcal{RDMH}$ be defined as the axiomatic extension of the logic 
$\mathcal{DMH}$ 
by the axiom:
\begin{equation} \tag{$\mathcal{ M}_2$}
  [\{\alpha \ \wedge (\sim(\sim \alpha)^*)\} \vee (\beta \vee \beta^{\ast})] \leftrightarrow_H (\beta \vee \beta^{\ast}).
  \end{equation}  
 \end{definition}
 
The following corollary is immediate from Theorem \ref{T5.9}  and Theorem \ref{T_DMH}.

\begin{Corollary}  \label{alg_semantics3}  
The logic 
$\mathcal{RDMH}$ is BP-algebraizable with the variety
   $\mathbb{RDMH}$ as its equivalent algebraic semantics. 
\end{Corollary}

The following corollary follows from 
Corollary \ref{C2.7},  Corollary \ref{C1},  Corollary \ref{alg_semantics1},  Corollary \ref{alg_semantics2}  and Corollary \ref{alg_semantics3}.                   

\begin{Corollary}
\begin{thlist}
\item $|\mathbf{Ext}(\mathcal{RDPCH})| = 2^{\aleph_0}$. 
  \item $|\mathbf{Ext}(\mathcal{RPCH}^d)|  = 2^{\aleph_0}$, 
\item $|\mathbf{Ext}(\mathcal{RDMH})|  = 2^{\aleph_0}$. 
\end{thlist} 
\end{Corollary}

\begin{remark}
We conclude the paper with the following observation:  Since the varieties $\mathbb{RDBLP}$ and $\mathbb{RDBLH}$ are
both term-equivalent to the variety $\mathbb{RDPCH}$, we could consider the logic $\mathcal{RDPCH}$ also as the logic of $\mathbb{RDBLP}$ as well as the logic of $\mathbb{RDBLH}$.  Similarly, since the varieties $\mathbb{RDMP}$, $\mathbb{RDMDBLP}$ and $\mathbb{RDMDBLH}$ are both term-equivalent to the variety $\mathbb{RDMH}$, we could consider the logic $\mathcal{RDMH}$ also as the logic for $\mathbb{RDMP}$, $\mathbb{RDMDBLP}$ or for $\mathbb{RDMDBLH}$.  
\end{remark}  

\medskip
\section*{Acknowledgements} 

 The first author wants to thank the institutional support of CONICET (Consejo Nacional de Investigaciones Cient\'ificas y T\'ecnicas) and Universidad Nacional del Sur.  




\end{document}